\documentclass[12pt]{amsart}
\usepackage{amssymb}
\usepackage{amsfonts}
\usepackage{amssymb,latexsym}
\usepackage{enumerate}
\usepackage{mathrsfs}

\usepackage{graphicx}
\usepackage[all]{xy}

\newtheorem{thm}{Theorem}[section]

\newtheorem{prop}[thm]{Proposition}
\newtheorem{mprop}[thm]{Main Proposition}

\newtheorem{rem}[thm]{Remark}

\theoremstyle{definition}
\newtheorem{definition}[thm]{Definition}

\theoremstyle{remark}

\begin{document}

\title[{\normalsize {\large {\normalsize {\Large {\LARGE }}}}}  conjectures on the representations of  modular Lie algebras
]{ conjectures on the representations of  modular Lie algebras}

\author{Kim YangGon }

\address{emeritus professor
 Department of Mathematics,Jeonbuk National University, Republic of Korea.}

\

\

\email{ kyk1.chonbuk@hanmail.net }


 \begin{abstract}

We have already seen simple representations of modular Lie algebras of $A_l$-type and $C_l$-type. We shall further investigate simple representations of $B_l$ type, which turn out to be very similar in methodology as those types except for roots. So  we may consider some conjectures relating to the representations of  classical type modular Lie algebras.

\end{abstract}

\maketitle

\section{introduction}

\

\

\large{The representation theories of simple Lie algebras for chacteristic zero of the ground field $F$ are almost done in case $F$ is algebraically closed, whereas those of simple Lie algebras for prime characteristic are under way.\newline

 The former ones belong to  nonmodular representation theory and the latter ones belong to modular representation theory of Lie algebras.\newline

We intend to investigate in section 2  of this paper the representations of modular Lie algebras of $B_l$-type following the ways of $A_l$-type and $C_l$-type.\newline

 Next  in section 3, we give some definitions such as Lee's basis, Park's Lie algebra, and  Hypo - Lie algebra, which are related to the representaton of  modular Lie algebras.\newline

 Afterwards we pose  some conjectures for simple Lie algebras of classical types in the last section 4 because those three kinds of representations mentioned above  are very similar  each other in methodology.\newline

 We assume throughout  that the ground field $F$ is algebraically closed unless otherwise stated.

\

\section{simple nonrestricted representations of $B_l$-type Lie algebra}

\

\

We are well aware that the orthogonal Lie algebra of $B_l$- type of rank $l$, i.e., the $B_l$-type Lie algera over $\mathbb{C}$ has its root system $\Phi$=$\{\pm\epsilon_i $(of squared lengh 1);  $\pm(\epsilon_i\pm\epsilon_j )$ (of squared length 2),$1\leq i \neq j \leq l $ \},where $\epsilon_i , \epsilon_j $ are linearly independent unit vectors in $\mathbb{R}^l$ with $l \geq 2$.The base of $\Phi$ equals $\{\epsilon_1-\epsilon_2,\epsilon_2- \epsilon_3,\cdots ,\epsilon_{l-1}-\epsilon_l,\epsilon_l\}.   $\newline

Let $L$ be a  $B_l$-type simlpe Lie algebra over an algebraically closed field $F$ of  characteristic $p\geq 7$.

\begin{prop}
Let $\chi $ be a character of any irreducible $L$-module with $\chi(x_\alpha)\neq $ 0  for  some $\alpha \in \Phi$,where $x_\alpha$ is an element in the Chevalley basis of  $L$ such that $F x_\alpha + Fh_\alpha+ Fx_{-\alpha}= \frak {sl}_2(F)$ with $[x_\alpha,x_{-\alpha}]= h_\alpha$.\newline

 We then have that any irreducible $L$-module with character $\chi$ is of dimension  $p^m=p^{n-l\over 2}$,where $n= dim L= 2m +l $ for a CSA  H with $dim H =l$.
\end{prop}

\begin {proof}

We meet with 2 cases of root length.\newline

(I) Suppose that $\alpha$ is  a short root. Since all roots of a given length are conjugate under the Weyl group of $\Phi$,we may put $\alpha$=$\epsilon_1$ without loss of generality.\newline

Let us put $B_i$=$b_{i1}$$h_{\epsilon_1 -\epsilon_2}$ +\newline
$b_{i2}h_{\epsilon_2-\epsilon_3}$+ $\cdots$ +$b_{i,l-1}h_{\epsilon_{l-1}-\epsilon_l}$ + $b_{il}h_{\epsilon_l}$ for $i=1,2,\cdots,2m $ ,where ($b_{i1},\cdots,b_{il})
\in F^l$ are chosen so that  arbitrary ($l+1)-B_i$'s are linearly independent in $\mathbb P^l(F)$,the $\frak B$ below becomes  an $F$-linearly independent set in $U(L)$ if necessary and  $x_\alpha B_i$ $\not \equiv B_i x_\alpha $ with $\alpha =\epsilon_1$.\newline

We search for a basis of $U(L)/\frak M_\chi$ ,where $\frak M_\chi$ is the kernel of the irreducible representation $\rho _{\chi}: L  \rightarrow \frak {gl}(V) $ of the restricted Lie algebra ($L ,[p])$ associated  with any given irreducible $L$-module $V$ with a character $\chi$.\newline

In $U(L)/\frak M_\chi$ we give a basis as $\frak B$:=$\{(B_1 +A_{\epsilon_1})^{i_1}\otimes (B_2 + A_{-\epsilon_1})^{i_2}\otimes(B_3+ A_{\epsilon_1 -\epsilon_ 2})^{i_3}\otimes(B_4+ A_{-(\epsilon_1-\epsilon_2) })^{i_4}\otimes \cdots\otimes (B_{2l}+A_{-(\epsilon_{l-1}-\epsilon_l)})^{i_{2l}}\otimes (B_{2l+1}+ A_{\epsilon _l})^{i_{2l+1}}\otimes(B_{2l+2}+A_{-\epsilon_l})^{i_{2l+2}}\otimes_{j=2l+3}^{2m}(B_j+ A_{\alpha_j})^{i_j}   \}$ for 0 $\leq i_j \leq p-1$,\newline

 where we put \newline
$A_{\epsilon_1}= x_{\epsilon_1}$, \newline
$A_{-\epsilon_1}=c_{-\epsilon_1}+ (h_{\epsilon_1} +1)^2 +
4x_{-\epsilon_1}x_{\epsilon_1}, \newline
A_{-\epsilon_1\pm \epsilon_2}= x_{\epsilon_1\pm \epsilon_2}(c_{-\epsilon_1\pm \epsilon_2}+x_{-\epsilon_1\pm \epsilon_2}x_{-(-\epsilon_1\pm \epsilon_2)}\pm x_{\pm \epsilon_2}x_{-(\pm \epsilon_2)}\pm x_{\epsilon_1 \pm \epsilon_2}x_{-(\epsilon_1 \pm \epsilon_2)})$,\newline
$A_{-\epsilon_1 \pm \epsilon_j}=x_{-\epsilon_2\pm \epsilon_j}(c_{\epsilon_1 \pm \epsilon_j}+x_{(\pm \epsilon_j-\epsilon_1)}
x_{-(\pm \epsilon_j -\epsilon_1)}\pm x_{\pm\epsilon_j}x_{-(\pm \epsilon_j)}\pm x_{\epsilon_1\pm \epsilon_j}x_{-(\epsilon_1\pm \epsilon_j)}),\newline
A_{\pm\epsilon_2}=x_{\epsilon_3\pm \epsilon_2}^2 (c_{\pm\epsilon_2}+x_{\epsilon_2}x_{-\epsilon_2}\pm x_{\epsilon_1 +\epsilon_2}x_{-(\epsilon_1 +\epsilon_2)}\pm x_{\epsilon_2-\epsilon_1}x_{\epsilon_1-\epsilon_2})$,\newline
$A_{\epsilon_j}=x_{\epsilon_2+ \epsilon_j}(c_{\epsilon_j}+ x_{\epsilon_j}x_{-\epsilon_j}\pm x_{\epsilon_1+\epsilon_j}x_{-(\epsilon_1 + \epsilon_j)}\pm \newline
 x_{\epsilon_j- \epsilon_1}x_{\epsilon_1-\epsilon_j})$, \newline
$A_{-\epsilon_j}=x_{\epsilon_2-\epsilon_j}(c_{-\epsilon_j}+ x_{-\epsilon_j}x_{\epsilon_j}\pm x_{\epsilon_1- \epsilon_j}x_{-(\epsilon_1-\epsilon_j)}\pm x_{-\epsilon_j-\epsilon_1}x_{\epsilon_1+ \epsilon_j} )  $,\newline

with the sign chosen so that  they commute with $x_\alpha$ and with $c_\beta \in F$ chosen so that $A_{-\epsilon_1}$ and parentheses(             ) are invertible.\newline

For any other root $\beta$,  we put $A_\beta={x_\beta}^2$  or $ x_\beta^3 $ if possible.\newline
Otherwise we make use of the parentheses(      )again used for designating $A_{-\beta}$. So in this case we put $A_\beta = { x_\gamma}^2$ or ${x_\gamma}^3$ attached to these (        ) so that  $x_\alpha$ may commute with $A_\beta$.\newline

We are going to show that $\frak B$ is a basis in $U(L)/\frak M_\chi$. \newline
It is not difficult to see that  $\frak B$ is a linearly independent set over $F$ in $U(L)$ by virtue of P-B-W theorem. Moreover $A_\beta \notin \frak M_\chi$ for any $ \beta \in \Phi$(see detailed proof below).\newline

We intend to prove that a nontrivial linearly dependence equation leads to absurdity.\newline

Suppoe that we have a  dependence equation which is of least degree with respect to $h_{\alpha_j}\in H$  and the number of whose highest degree terms is also least. If it is conjuated by $x_\alpha$, then there arises a nontrivial dependence equation of lower degree than the given one, which is absurd.\newline

Otherwise it reduces to one of the following forms:  \newline
 (i) $x_{\epsilon_j}K+ K' \in \frak M_\chi$,\newline
(ii)$x_{-\epsilon_j}K + K'\in \frak M_\chi$,\newline
(iii)$x_{\epsilon_j + \epsilon_k}K + K'\in \frak M_\chi$,\newline
(iv)$x_{-\epsilon_j -\epsilon_k}K+ K'\in \frak M_\chi$,\newline
(v)$x_{\epsilon_j-\epsilon_k}K+ K'\in \frak M_\chi$,\newline
where $K$ and $K'$ commute with $x_\alpha=x_{\epsilon_1}$.\newline

Because (i),(ii) reduce to (iii),(iv) respectively by applying $ad x_\alpha$,we have only to consider (iii),(iv) and (v) .\newline

As for the case (v) we deduce successively,
$x_{-\epsilon_j}( x_{\epsilon_j-\epsilon_k}K+ K')\in \frak M_\chi \Rightarrow (x_{-\epsilon_k}+ x_{\epsilon_j- \epsilon_k}x_{-\epsilon_j})K + x_{-\epsilon_j}K'\in \frak M_\chi \Rightarrow$ by $adx_{\epsilon_1}$,$(x_{\epsilon_1-\epsilon_k}+ x_{\epsilon_j -\epsilon_k}x_{\epsilon_1-\epsilon_j})K + x_{\epsilon_1- \epsilon_j}K'\in \frak M_\chi$ for $j,k\neq1$ $\Rightarrow x_{\epsilon_1- \epsilon_k}K+ K'' \in \frak M_\chi $ for some $K''$ commuting with $x_\alpha$ and being compared to the start equation.\newline

In this case we deduce $x_{\epsilon_k-\epsilon_1}x_{\epsilon_1-\epsilon_k}K+ x_{\epsilon_k-\epsilon_1}K'' \in \frak M_\chi \Rightarrow (h_{\epsilon_k- \epsilon_1}+ x_{\epsilon_1- \epsilon_k}x_{\epsilon_k- \epsilon_1})K+ x_{\epsilon_k-\epsilon_1}K''\in 
 \frak M_\chi\Rightarrow (x_{\epsilon_1}+ x_{\epsilon_1- \epsilon_k}x_{\epsilon_k})K+ x_{\epsilon_k}K''\in \frak M_\chi4$ by $adx_{\epsilon_1}$
 $\Rightarrow x_{\epsilon_1}K+ x_{\epsilon_1-\epsilon_k}x_{\epsilon_k}K+ x_{\epsilon_k}K''\in \frak M_\chi \Rightarrow $

we may put $x_{\epsilon_1}K+ K'''\in \frak M_\chi$ for some $K'''$commuting with $x_\alpha$ modulo $\frak M_\chi$ because $x_{\epsilon_1}K$ commute with $x_\alpha= x_{\epsilon_1}$.\newline

 We may also  deduce a similar form $x_{\epsilon_1}K+ K'''\in \frak M_\chi$ from the types (iii) and (iv). 
So we have only to prove that such a form is absurd.\newline

From this we get $x_{-\epsilon_1}x_{\epsilon_1}K+ x_{-\epsilon_1}K'''\in \frak M_\chi$ ,so we get $4^{-1}\{w- (h+ 1)^2\}K+ x_{-\epsilon_1}K'''\in \frak M_\chi.$ Here $w= (h+ 1)^2 + 4 x_{-\epsilon_1}x_{\epsilon_1}$ is contained in the center of $U(\frak {sl}_2(F))$.\newline

  If $x_{-\epsilon_1}^{p-1}\equiv c,$ then multiplying this last equation by $x_{-\epsilon_1}^{p-1} $ gives rise to $4^{-1}x_{-\epsilon_1}^{p-1}\{w- (h+ 1)^2\}K+ cK'''\equiv 0. $
Multplying this equation by $x_{\epsilon_1}^{p-1}$,we have $4^{-1}x_{\epsilon_1}^{p-1}x_{-\epsilon_1}^{p-1}\{w- (h+ 1)^2\}K+ cx_{\epsilon_1}^{p-1}K'''\equiv 0.$By making use of $w$,we have  an equation of the form \newline

(a polynomial of degree $\geq 1$ with respect to $h)K+ cx_{\epsilon_1}^{p-1}K'''\equiv 0$. Finally consecutive conjugation and subtraction by $x_{\epsilon_1}$ gives rise to $K\in \frak M_\chi,$which is absurd.\newline

(II) Suppose that  $\alpha$ is a long root; then we may put $\alpha=\epsilon_1-\epsilon_2$ since all roots of the same length are conjugate under the Weyl group of $\Phi$.\newline

Similarly as in (I), we put $B_i:=$the same as in (I) except that  $\alpha=\epsilon_1 -\epsilon_2$ this time instead of $\epsilon_1$.\newline

We claim we have a basis $\frak B$$:= \{(B_1+ A_{\epsilon_1- \epsilon_2})^{i_1}\otimes (B_2+ A_{-(\epsilon_1-\epsilon_2)})^{i_2}\otimes \cdots \otimes(B_{2l-2}+ A_{-(\epsilon_{l-1}-\epsilon_l)})^{i_{2l-2}}\otimes (B_{2l-1}+ A_{\epsilon_l})^{i_{2l-1}}\otimes (B_{2l}+ A_{-\epsilon_l})^{i_{2l}}\otimes (\otimes_{j=2l+1}^{2m}(B_j+ A_{\alpha_j})^{i_j}),0\leq i_ j\leq p-1 \} $,\newline

where we put  \newline
$A_{\epsilon_1-\epsilon_2}= x_\alpha= x_{\epsilon_1- \epsilon_2},\newline
 A_{\epsilon_2- \epsilon_1}=c_{\epsilon_2- \epsilon_1}+ (h_{\epsilon_1- \epsilon_2}+ 1 )^2+ 4x_{\epsilon_2-\epsilon_1}x_{\epsilon_1- \epsilon_2} $,\newline
$A_{\epsilon_2\pm \epsilon_3}=     x_{\pm \epsilon_3}(c_{\epsilon_2\pm \epsilon_3}+ x_{\epsilon_2\pm \epsilon_3}x_{-(\epsilon_2\pm \epsilon_3)}\pm x_{\epsilon_1\pm \epsilon_3}x_{-(\epsilon_1\pm \epsilon_3)})$, \newline
$A_{\epsilon_2\pm \epsilon_k}= x_{\epsilon_3\pm \epsilon_k}
(c_{\epsilon_2\pm \epsilon_k}+ x_{\epsilon_2\pm \epsilon_k}x_{-(\epsilon_2\pm \epsilon_k)}\pm x_{\epsilon_1\pm \epsilon_k}x_{-(\epsilon_1\pm \epsilon_k)})  $if $k\neq 1$, \newline
$ A_{\epsilon_2}= x_{\epsilon_1}( c_{\epsilon_2}+ x_{\epsilon_2}x_{-\epsilon_2}\pm x_{\epsilon_1}x_{-\epsilon_1})$,  \newline
                 $A_{-\epsilon_1}= x_{-\epsilon_2}(c_{-\epsilon_1}+ x_{-\epsilon_1}x_{\epsilon_1}\pm x_{-\epsilon_2}x_{\epsilon_2})$,\newline 
$A_{-(\epsilon_1\pm \epsilon_3)}= x_{-(\pm \epsilon_3)}(c_{-(\epsilon_1\pm \epsilon_3)}+ x_{\epsilon_2\pm \epsilon_3}x_{-(\epsilon_2\pm \epsilon_3)}\pm x_{\epsilon_1 \pm\epsilon_3}x_{-(\epsilon_1\pm \epsilon_3)}) $, $A_{-(\epsilon_1\pm \epsilon_k)}= x_{-(\epsilon_3\pm \epsilon_k)}(c_{-(\epsilon_1\pm \epsilon_k)}+ x_{\epsilon_2\pm \epsilon_k }x_{-(\epsilon_2\pm \epsilon_k)}\pm x_{\epsilon_1\pm \epsilon_k}x_{-(\epsilon_1\pm \epsilon_k)}) $, $A_{\epsilon_l}= x_{\epsilon_l}^2 ,  $\newline $A_{-\epsilon_l}=x_{-\epsilon_l}^2  $, \newline

with the signs chosen so that they commute with $x_\alpha$ and with $c_\beta\in F$ chosen so that $A_{\epsilon_2-\epsilon_1}$ and parentheses are invertible.\newline

For any other root $\beta$, we put $A_\beta= x_\beta^2 $ or $x_\beta^3 $ if possible.\newline

Otherwise we make use of the parentheses(      ) again used for designating $A_{-\beta}$. So in this case we put $A_\beta= x_\gamma^2 $       
 or $ x_\gamma^3  $ attached to these (      ) so that  $x_\alpha$ may commute with $A_\beta$.\newline

We intend to show that $\frak B$ is a basis in $U(L)/\frak M_\chi$.
We may see easily that $\frak B$ is a linearly independent set over $F$ in $U(L)$ by virtue of P-B-W theorem.  Furthermore for any $\beta\in \Phi$, $A_\beta\notin \frak M_\chi$(see detailed proof beow).\newline

 We shall show that a nontrivial linearly dpendence equation leads to absurdity. We assume that we have a  dependence equation which is of least degree with respect to $h_{\alpha_j}\in H$ and the number of whose highest degree terms is also least.\newline

If it is conjugated by $x_\alpha=x_{\epsilon_1-\epsilon_2}$, then we get a nontrivial dependence equation of lower degree than the given one contravening our assumption.\newline

Otherwise it reduces to one of the following forms:\newline
(i)$x_{\epsilon_j}K+ K'\in \frak M_\chi$,\newline
(ii)$x_{-\epsilon_j}K+ K'\in \frak M_\chi$,\newline
(iii)$x_{\epsilon_j+ \epsilon_k}K+ K'\in \frak M_\chi$,\newline
(iv)$x_{-\epsilon_j- \epsilon_k}K+ K'\in \frak M_\chi$,\newline
(v)$x_{\epsilon_j-\epsilon_k}K+ K'\in \frak M_\chi$,\newline 
(vi)$x_{\epsilon_k-\epsilon_j}K+ K'\in \frak M_\chi$,\newline
where $K$ and $K'$ commute with $x_\alpha=x_{\epsilon_1- \epsilon_2}$.\newline

As for the case (i),(ii), these may easly be changed to the forms from (iii) to (vi). Moreover the cases (iii),(iv),(vi) are analogous to the case (v).\newline

Hence it suffices to conider only the case (v). However the proof leading to absurdity is very similar in methodology as that of the case(I) above or that of the case  for $C_l$-type Lie algebra in proposition4.1 [3].
\end{proof}

\
\section{Definitions}
\
\

We give here some definitions related to some conjectures  in the next section.

\begin{definition}

Let $L$ be  a  finite dimensional Lie algebra over an algebbraically  closed  field $F$
 of characteristic $p>0$  and let  $H$ be a  Cartan subalgebra(abb.CSA) of $L.$ \newline

We shall call $B$ a $Lee$'s basis of the algebra $U(L)$/$\frak M$ for a maximal ideal 
$\frak M$  of  $U(L)$ if $B$ satisfies the following properties:\newline

(i) $\dim_{F}U(L)/\frak M =[Q(U(L)):Q(Z)]=p^{2m}$
 for the center $Z$ of the universal enveloping algebra $U(L)$ of $L,$ where $Q(U(L))$ and $Q(Z)$ are quotient algebras of $U(L)$ and $Z$ respectively;\newline

(ii) $B$ is a basis of $U(L)/\frak M$ over $F$ and is of  the form 
$\{(B_1 +A_1)^{i_1}\otimes(B_2 +A_2)^{i_2}\otimes \cdots \otimes(B_{2m} +A_{2m})^{i_{2m}}\}$,
where $0\leq i_{j}\leq p-1$
and $B_i$'s are elements in the subalgebra of $U(L)$ generated by 1 and $H$, and $A_i$'s are some elements in $U(L).$
\end{definition}

\begin{definition}

Suppose that $L$ is an indecomposable Lie algebra over a field $F$ of nonzero characteristic $p.$

If $U(L)$/$\frak M$ has a Lee's basis for all maximal ideals $\frak M$  except for a finite number of them up to isomorphism, we shall call $L$ a $ Park$'s Lie algebra.

 \end{definition}

\begin{definition}
By a $Hypo$ - Lie algebra,  we shall mean a sub-Lie algebra of some simple Lie algebra which also becomes a $Park$'s Lie algebra.

\end{definition}

Earlier in 1967, two scholars I.R.Shafarevich  and  Rudakof  observed such a fact  regarding irreducible modules of $L= \frak {sl}_2(F)$ [5].
They proved that  for any nonzero character $\chi$ any irreducible L-module over $F$ associated with $\chi$  is of dimension $p$, where $F$ is an algebraically closed field of characteristic $p\geq3.$\newline

Also they proved that for zero character $\chi$=0, the number of nonisomorphic ireducible $L$-modules is finite.
For this proof the center of $U(L)$ is crucial  according to them.

Now we might as well give a nontrivial example of $Hypo$- Lie algebra.\newline

Let $L$ be a Park's Lie algebra as a simple Lie algebra contained in another classical simple Lie algebra $L'$ over an algebraically closed field $F$ of characteristic $p\geq7$.
Let $H,H'$ be CSA's of $L,L'$ respectively satisfying $L \cap H' = H$. \newline

Given a Chevalley basis $\mathfrak B_c$ with respect to $H'$ of $L'$, we let  
$h_{\alpha}$$\in$ $(H'- H) \cap \frak B_c$ for a certain root $\alpha$. Suppose that  for this  $\alpha$, $L_{H}:= L    \cup$ ${\{{h_{\alpha}}}\}$ is indecomposable.
\newline

Then we may see without difficulty that $L_{H}$ becomes a $Hypo$ - Lie algebra [2].\newline

 We are not  sure for now whether or not we may extend such a fact to the Cartan type Lie algebras. According to [6], there are only 2 kinds of finite dimensional 
simple Lie algebras over an algebraically closed field $F$ of characteristic $p \geq 7$ , namely Classical type or Cartan type simple Lie algebras.

\

\section{ Conjectures}

\

\

We found out some counter examples to the nonrestricted representation theory of  $C_{l}$-type Lie algebras in [3].\newline

 Furthermore
a  Lie algebraist J{\"o}rg Feldvoss reviewed the paper [1] in the affirmative. He agreed tacitly to the fact that any $A_l$ type modular Lie algebra with $l\geq1$ becomes a Park's Lie algebra and so a Hypo Lie algebra over an algebraically closed field $F$ of  characteristic $p\geq7$.\newline

 In other words, for any nonzero character $\chi$  we have  
 $W_{\chi}(L)$=
$V_{\chi}(L)$    for $L= A_l$- type simple Lie algebras  except for a finite number of  restricted  irreducible $L-$modules, where the left hand side is a Weyl module of $L$  and the right hand side is a Verma module of $L$  associated with a character  $\chi.$ \newline

So we could give a conjecture as follows, motivated by this fact in [1] and section2 of this paper and references [2],[3]:\newline

[$\textbf {CONJECTURE} $]  Suppose that $L$ is a Lie algebra of classical type over an algebraically closed field $F$ of characteristic $p\geq7$, and that $L'$ is another finite dimensional simple
Lie algebra containing $L$ with a CSA $H'$ such that $H' \cap L $ is a CSA of $L$.\newline

Then (i)$L$ + $H'$ is a subalgebra of $L'$;

(ii) if $L$ +$H'$ is an  indecomposable subalgebra of $L'$, then $L$ + $H'$ becomes a  $Hypo$- Lie algebra 
and the maximal dimension of irreducile ($L$ +$H'$)-modules equals that of irreducible $L$-modules.

\

\bibliographystyle{amsalpha}

\end{document}